\documentclass[12pt]{article}
\usepackage{amsmath, amsthm, amsfonts, amssymb}
\usepackage{mathtools} % For dcases
\usepackage[utf8]{inputenc}
\usepackage[all,cmtip]{xy}
\usepackage{xcolor}
\usepackage{graphicx}
\graphicspath{ {./images/} }

\usepackage{csquotes}

\usepackage{eucal}

\usepackage{tikz}

\usetikzlibrary{decorations.markings}

\theoremstyle{plain}
\newtheorem{thm}{Theorem}[section]

\newtheorem{prop}[thm]{Proposition}

\theoremstyle{definition}
\newtheorem{defn}[thm]{Definition}%[section]
\newtheorem{rem}[thm]{Remark}

\numberwithin{equation}{section}

\newcommand{\squarething}[1]{
  \noindent
  \begin{center}
  \framebox{
    \vbox{
      \vspace{4mm}
      \hbox to 5.78in { {\Large \hfill #1  \hfill} }
      \vspace{4mm}
    }
  }
  \end{center}
  \vspace*{4mm}
}

\advance \topmargin by -\headheight
\advance \topmargin by -\headsep
\evensidemargin \oddsidemargin
\DeclareMathOperator{\coord}{Coord}
\DeclareMathOperator{\id}{id}

\DeclareMathOperator{\en}{End}

\DeclareMathOperator{\Hom}{Hom}

\DeclareMathOperator{\res}{Res}

\DeclareMathOperator{\aut}{Aut}

\DeclareMathOperator{\img}{im}
\DeclareMathOperator{\coker}{coker}
\DeclareMathOperator{\Ext}{Ext}
\DeclareMathOperator{\der}{Der}

\newcommand{\dualfactor}{\mathcal{R}}

\newcommand{\D}{\Delta}

\newcommand{\C}{\mathbb{C}}

\newcommand{\Z}{\mathbb{Z}}
\newcommand{\bbP}{\mathbb{P}}

\newcommand{\g}{\mathfrak{g}}

\newcommand{\CA}{\mathcal{A}}

\newcommand{\CD}{\mathcal{D}}

\newcommand{\CM}{\mathcal{M}}
\newcommand{\CP}{\mathcal{P}}

\newcommand{\CV}{\mathcal{V}}

\newcommand{\OO}{\mathcal{O}}

\newcommand{\vac}{\mathbf{1}}

\begin{document}

%\maketitle

\begin{center}
{\LARGE \bf Chiral homology of the projective line and extensions of vertex algebra modules} \par \bigskip

\renewcommand*{\thefootnote}{\fnsymbol{footnote}}
{\normalsize
Thadeu Henrique Cardoso, Jethro van Ekeren, \\
Juan Guzman and Reimundo Heluani
}

\par \bigskip
{\footnotesize Instituto de Matem\'{a}tica Pura e Aplicada \\
Rio de Janeiro, RJ, Brazil}

\par \bigskip

%\vspace*{10mm}

\end{center}

\vspace*{10mm}

\noindent
\textbf{Abstract.} The purpose of this note is to establish an isomorphism from the vector space of extensions between two modules over a vertex algebra, to an appropriate first chiral homology of one dimensional projective space with coefficients in the corresponding chiral algebra.

%%%%%%%%%%%%%%%%%%%%%%%%%%%%%%%%%%%%%%%%%%%%%%%%%%%%%%%%%%%%%%%%%%%%%%%%%%%%%%%%
%%%%%%%%%%%%%%%%%%%%%%%%%%%%%%%%%%%%%%%%%%%%%%%%%%%%%%%%%%%%%%%%%%%%%%%%%%%%%%%%
%%%%%%%%%%%%%%%%%%%%%%%%%%%%%%%%%%%%%%%%%%%%%%%%%%%%%%%%%%%%%%%%%%%%%%%%%%%%%%%%
%%%%%%%%%%%%%%%%%%%%%%%%%%%%%%%%%%%%%%%%%%%%%%%%%%%%%%%%%%%%%%%%%%%%%%%%%%%%%%%%
%%%%%%%%%%%%%%%%%%%%%%%%%%%%%%%%%%%%%%%%%%%%%%%%%%%%%%%%%%%%%%%%%%%%%%%%%%%%%%%%

\section{Introduction}

Let $\g$ be a finite dimensional semisimple Lie algebra over the complex number field $\C$. A classic theorem of Weyl asserts complete reducibility of all finite dimensional $\g$-modules. The proof has two ingredients. The first is a general construction of homological algebra, namely an isomorphism from the vector space $\Ext^1(C, A)$ of equivalence classes of extensions of $C$ by $A$ (where $A$ and $C$ are $\g$-modules), to a cohomology space $H^1(\g, \Hom(C, A))$ which may be described explicitly in terms of Chevalley-Eilenberg cocycles. The second is a proof of vanishing of cohomology using the semisimplicity hypothesis.

Vertex algebras are certain, usually infinite dimensional, algebraic objects with important applications in representation theory and theoretical physics. Homology and cohomology theories of vertex algebras have been introduced and studied by various authors, see \cite{Huang2014}, \cite{Huang-Fei} and \cite{Liberati2017} for example, as well as the series of articles starting with \cite{BDHK2019}. In these approaches, first cohomology $H^1(V, M)$ recovers the appropriate notion of outer derivation of a module $M$ over a vertex algebra $V$, parallel to the case of Lie algebras.

Meanwhile Beilinson and Drinfeld \cite{BD} have introduced the notion of chiral algebra over an algebraic curve $X$, as well as a definition of chiral homology of $X$ with coefficients in a chiral algebra $\CA$ over $X$. Chiral homology generalises the construction of conformal blocks, which is in turn recovered in degree $0$. A conformal vertex algebra $V$ gives rise to a chiral algebra $\CA_V$ over each algebraic curve $X$, see \cite{FBZ}. A key feature of chiral homology is that multiple modules can be `inserted' at distinct points of $X$.

In this note we prove the following theorem.
\begin{thm}\label{thm:main.1}
Let $V$ be a conformal vertex algebra and $A$ and $C$ two admissible $V$-modules. Then
\[
\Ext^1(C, A) \cong H_{1}^{\textup{ch}}(\bbP^1, \CA_{V},\CM_{C,0},\CM_{A^{\vee},\infty})^{*}.
\]
\end{thm}
We remark that Theorem \ref{thm:main.1} actually holds more generally for M\"{o}bius vertex algebras, see Remark \ref{rem:mobius} below. As an essentially immediate corollary of Theorem \ref{thm:main.1} we deduce
\begin{thm}\label{thm:main.2}
Let $V$ be a conformal vertex algebra. Then $V$ is rational if and only if
\[
H^{\textup{ch}}_1(\bbP^1, \CA_V, \CM_{C,0},\CM_{A^{\vee},\infty}) = 0
\]
for all admissible $V$-modules $A$ and $C$.
\end{thm}

\emph{Acknowledgements:} The authors would like to thank Tomoyuki Arakawa for helpful conversations. TC has been supported by a PhD stipend from Capes, JvE by grants Serrapilheira -- 2023-0001, CNPq 306498/2023-5 and FAPERJ 201.445/2021, JG by an INCTMat postdoc, and RH by CNPq 311451/2023-3 and a CNE grant from FAPERJ.

\section{Background on vertex algebras}

A standard reference for this material is \cite{KacVA}, see also \cite{FHL} and \cite{FBZ}. A quantum field on a vector space $V$ is a formal series
\[
A(z) = \sum_{n \in \Z} a_{(n)} z^{-n-1}
\]
with coefficients $a_{(n)} \in \en(V)$, such that for each $b \in V$ one has $A(z)b \in V((z))$ the vector space of formal Laurent series.

Let $\C(z, w)$ denote the fraction field of the polynomial algebra $\C[z, w]$. We define $i_{z, w} : \C(z, w) \rightarrow \C((z))((w))$ to be the unique morphism extending the inclusion $\C[z, w] \subset \C((z))((w))$, and $i_{w, z}$ similarly.

We use the symbol $\res_{z=0} f(z)$ to denote the coefficient of $z^{-1}$ of a formal power series $f(z) \in U[[z^{\pm 1}]]$, and the symbol $\res_{z_1=z_2} f(z_1, z_2)$ to denote the coefficient of $(z_1-z_2)^{-1}$ of the image of $f(z_1, z_2)$ under the natural map
\[
U[[z_1, z_2]][z_1^{-1}, z_2^{-1}, (z_1-z_2)^{-1}] \rightarrow U((z_2))[(z_1-z_2)^{-1}],
\]
obtained by expanding powers of $z_1 = z_2 + (z_1-z_2)$ in positive powers of $z_1-z_2$.

The $n^{\text{th}}$ product of quantum fields $A(w)$ and $B(w)$ on a vector space $V$ is
\[
A(w)_{(n)}B(w) = \res_{z=0} \left( A(z)B(w) i_{z, w} - B(w)A(z) i_{w, z} \right) (z-w)^n,
\]
which is well-defined and again a quantum field on $V$.

\begin{defn}\label{defn:va}
A vertex algebra consists of a vector space $V$, for each $a \in V$ a quantum field
\[
Y(a, z) = \sum_{n \in \Z} a_{(n)} z^{-n-1}
\]
on $V$, as well as a distinguished vector $\vac \in V$ (the vacuum vector) such that the following conditions are satisfied.
\begin{itemize}
\item For all $a, b, c \in V$ and $f(z_1, z_2) \in \C[z_1^{\pm 1}, z_2^{\pm 1}, (z_1-z_2)^{-1}]$ the following equality holds
\begin{align*}
\res_{z_2=0}\res_{z_1=z_2} & f(z_1,z_2)Y(Y(a,z_1 - z_2)b,z_2)c \\
= {} & \res_{z_1=0}\res_{z_2=0}f(z_1,z_2)Y(a,z_1)Y(b,z_2)c \\
&- \res_{z_1=0}\res_{z_2=0}f(z_1,z_2)Y(b,z_2)Y(a,z_1)c,
\end{align*}

\item $Y(\vac, z) = I_V$, and

\item for each $a \in V$ we have $Y(a, z)\vac \in V[[z]]$ and $Y(a, z)\vac|_{z=0} = a$.
\end{itemize}
\end{defn}
The translation operator $T \in \en(V)$ is recovered as follows: $Y(a, z)\vac = a + (Ta) z + \cdots$, and satisfies
\begin{align}\label{eq:T.property}
[T, Y(a, z)] = Y(Ta, z) = \partial_z Y(a, z)
\end{align}
for all $a\in V$.

\begin{defn}\label{defn:va.module}
Let $V$ be a a vertex algebra. A $V$-module is a vector space $M$ and for each $a \in V$ a quantum field
\[
Y_M(a, z) = \sum_{n \in \Z} a^M_{(n)} z^{-n-1}
\]
on $M$, such that the following conditions are satisfied.
\begin{itemize}
\item For all $a, b, \in V$ and $m \in M$, and $f(z_1, z_2) \in \C[z_1^{\pm 1}, z_2^{\pm 1}, (z_1-z_2)^{-1}]$ the following equality holds
\begin{align*}
\res_{z_2=0}\res_{z_1=z_2} & f(z_1,z_2)Y_M(Y(a,z_1 - z_2)b,z_2)m \\
= {} & \res_{z_1=0}\res_{z_2=0}f(z_1,z_2)Y_M(a,z_1)Y_M(b,z_2)m \\
&- \res_{z_1=0}\res_{z_2=0}f(z_1,z_2)Y_M(b,z_2)Y_M(a,z_1)m,
\end{align*}

\item $Y_M(\vac, z) = I_M$.
\end{itemize}
\end{defn}

\begin{rem}
The definition of $V$-module given above is one of two variants. In the other variant, an endomorphism $T \in \en(M)$ is also given, for which the obvious analogue of \eqref{eq:T.property} is required to hold. In this article we shall primarily be concerned with conformal vertex algebras \textup{(}Definition \ref{def:conformal.VA} below\textup{)} and their modules, and in this context there is no difference between the two definitions.
\end{rem}

\begin{defn}\label{def:conformal.VA}
A conformal vertex algebra is a vertex algebra $V$ together with a grading $V = \bigoplus_{n \in \Z_+} V_n$ and a conformal vector $\omega \in V_2$ whose associated quantum field $L(w) = \sum_{n \in \Z} L_n w^{-n-2}$ equips $V$ with a representation of some central charge $c$ of the Virasoro Lie algebra
\[
[L_m, L_n] = (m-n) L_{m+n} + \delta_{m, -n} \frac{m^3-m}{12} c I_V,
\]
compatible with the vertex algebra structure in the sense that (1) $L_0|_{V_n} = n I_{V_n}$ and (2) $L_{-1} = T$.
\end{defn}

\begin{defn}
Let $V$ be a conformal vertex algebra. A $V$-module $M$ is said to be admissible if it carries a $\Z_+$-grading $M = \bigoplus_{n \in \Z_+ } M_n$ such that $a_{(m)}(M_n) \subseteq M_{r+n-m-1}$ for all $m \in \Z$, $r,n \in \Z_+$ and $a \in V_{r}$. The vertex algebra $V$ is said to be rational if every admissible $V$-module is completely reducible.
\end{defn}

\begin{rem}\label{rem:exponentiate.actions.admissible}
If $V$ is a conformal vertex algebra, then any $V$-module $M$ is in particular a representation of the Virasoro Lie algebra via the quantum field $Y_M(\omega,z)=\sum_{n \in \Z} L_n z^{-n-2}$ {\cite[Prop. 4.1.5]{LL}}. If $M$ is admissible, then the action of the Lie subalgebra $\der_+(\C[[z]])$ topologically generated by the $L_k$ with $k \geq 1$ is locally nilpotent, since $L_k(M_n) \subseteq M_{n-k}$. So we can exponentiate this action to obtain an action of the Lie group $G_+=\aut_+(\C[[z]])$ on $M$.
However, the action of $\der(\C[[z]])=\C L_0 \oplus \der_+(\C[[z]])$ cannot always be exponentiated to an action of the bigger group $G=\aut(\C[[z]])$ on $M$. This can only be done if the eigenvalues of $L_0$ are integral.
\end{rem}

We denote by $\Ext^1(C, A)$ the set of isomorphism classes of extensions of $C$ by $A$, that is of short exact sequences
\begin{align*}
\xymatrix{
0 \ar@{->}[r] & A \ar@{->}[r] & B \ar@{->}[r] & C \ar@{->}[r] & 0,
}
\end{align*}
in the category of admissible $V$-modules. This set is naturally a vector space since the category of admissible $V$-modules is a $\C$-linear abelian category.

\begin{defn}\label{def:dual}
Let $V$ be a conformal vertex algebra and $M$ an admissible $V$-module. The contragredient $V$-module $M^\vee$ is by definition the restricted dual vector space $M^\vee = \bigoplus_{n \in \Z_+} M_n^*$, equipped with the $V$-action
\begin{align}\label{eq:contragredient.eq}
[Y(a, z) \varphi](m) = \varphi\left( Y(e^{z L_1} (-z^{-2})^{L_0} a, z^{-1}) m \right).
\end{align}
\end{defn}

\begin{rem}\label{rem:contragred.topology}
The contragredient module construction is usually applied to $V$-modules whose graded pieces are finite-dimensional, such as ordinary $V$-modules (see {\cite[Section 5.2]{FHL}}). For an arbitrary admissible $V$-module with possibly infinite-dimensional graded pieces we find that the contragredient dual $M^\vee$ is still well-defined and \eqref{eq:contragredient.eq} continues to furnish $M^\vee$ with the structure of a $V$-module. In this case, each dual space $M_n^*$ acquires the structure of a topological vector space such that annihilators of finite dimensional vector spaces of $M_n$ form a base of open neighbourhoods of $\{0\}$, and $M_n^*$ is linearly compact with this topology {\cite[Chapter 1.2]{Dieudonne}}. The action of $V$ on $M^\vee$ is now continuous.
\end{rem}

\section{Chiral algebras and chiral modules}\label{sec:chiral.algebras}

Let $X$ be a smooth complex algebraic curve and $V$ a conformal vertex algebra. Following Beilinson and Drinfeld \cite{BD}, a chiral algebra on $X$ is a right $\CD_X$-module together with a morphism
\[
\mu : j_*j^*(\CA \boxtimes \CA) \rightarrow \D_*(\CA)
\]
of right $\CD_{X^2}$-modules satisfying analogues of the skew symmetry and Jacobi identities of Lie algebras. Here $\D : X \rightarrow X^2$ and $j : X^2 \backslash X \rightarrow X^2$ are the inclusions of the diagonal and its complement. In general for a right $\CD_X$-module $\CA$ one has a canonical morphism
\begin{align}\label{eq:Sato.mor}
j_*j^*(\omega_X \boxtimes \CA) \rightarrow \D_*(\CA)
\end{align}
whose kernel is $\omega_X \boxtimes \CA$. Here $\omega_X$ is the canonical sheaf with its natural right $\CD_X$-module structure. Then $\CA = \omega_X$ with \eqref{eq:Sato.mor} for $\mu$ is an example of a chiral algebra on $X$. A chiral algebra $\CA$ is said to be unital if it comes with an inclusion $\omega_X \rightarrow \CA$ compatible with \eqref{eq:Sato.mor}.

For $V$ a conformal vertex algebra, Frenkel and Ben Zvi present in \cite{FBZ} a construction of a unital chiral algebra $\CA = \CA_V$ over $X$. As a $\CD_X$-module $\CA = \CV \otimes \omega_X$ is obtained via side-changing from a left $\CD_X$-module $\CV$. As an $\OO_X$-module, $\CV$ is the colimit of finite dimensional holomorphic vector bundles constructed from truncations $V_{\leq n}$ of $V$ via an associated bundle construction.

In particular the fibres of $\CV$ have the following description. Let $G = \aut(\C[[t]])$ and let $\coord_p$ denote the $G$-torsor of formal coordinates at the point $p \in X$. The conformal structure of $V$ endows it with a left $G$-action, and the fibre of $\CV$ at $p$ is the vector space
\[
\CV_p = \coord_p \times_G V.
\]
A choice of local coordinate at $p$ furnishes a linear isomorphism between $\CV_p$ and $V$. An element of $\CV_p$ may thus be written in the form $(z, a)$, and a section of $\CV$ over a trivialising open $U \subset X$ with coordinate $z$ may be written as a linear combination of terms of the form $f(z) (z, a)$ where by abuse of notation the value of $(z, a)$ at $p \in U$ is the element of $\CV_p$ represented by $(z-z(p), a)$.

A chiral $\CA$-module on $X$ is a right $\CD_X$-module $\CM$ and a morphism
\[
\alpha_{\CM} : j_*j^*(\CA \boxtimes \CM) \rightarrow \D_*(\CM)
\]
of $\CD_{X^2}$-modules satisfying a natural condition of compatibility with $\mu$. We are interested in chiral $\CA$-modules supported at a point $p \in X$. Such modules can be constructed in the following manner: if $M$ is an admissible $V$-module such that the eigenvalues of $L_0$ are integral, then we have a well-defined action of $G$ on $M$ by Remark \ref{rem:exponentiate.actions.admissible}. Now a construction parallel to that of $\CA_V$ yields a chiral $\CA_V$-module $\CM_M$. Let $p$ be a point of $X$ and $j_p : X \backslash p \rightarrow X$ the inclusion of its complement. We may now form the $\CA_{V}$-module $\CM_{M,p}=\coker\left(\CM_{M} \rightarrow j_{p*}j_{p}^{*}\CM_{M}\right)$.

The construction described above does not work for more general $V$-modules $M$, as the $\CD_X$-module $\CM_M$ is not well-defined. We do have, however, the following construction from \cite{FBZ}. Let $M$ be any admissible $V$-module, with possibly non-integral eigenvalues of $L_0$. Then we still have an action of the subgroup $G_+ = \aut_+(\C[[t]])$ of $G$ on $M$, and we may use it to construct
\[
\CM_{M, p, \xi} = \coord_{p, \xi} \times_{G_+} M,
\]
where $\coord_{p, \xi} \subset \coord_p$ consists of those coordinates with first order differential $\xi \in T_p X$. Then $\CM_{M, p, \xi}$ acquires an action $\alpha_{\CM_{M, p, \xi}}$ of local sections of $\CA_V$ {\cite[Theorem 7.3.10]{FBZ}}. To explicitly describe this action, we make a choice of local coordinate at $z$ with first jet equal to $\xi$. This induces an identification of $\CM_{M, p, \xi}$ with $M$, and a local section $\sigma = f(z) (z, a) \, dz$ of $\CA_V$ then acts as
\begin{align}\label{eq:point.module.action}
\alpha_{\CM_{M, p, \xi}}(\sigma) m = \res_{z=0} f(z) Y_M(a, z) m.
\end{align}

In the sequel we shall be concerned with $X = \bbP^1$ the Riemann sphere, which we equip with its two standard charts $(U_1, z)$ and $(U_2, x)$, each isomorphic to $\C$, and coordinates related by $x = 1/z$ on the intersection. We consider $\CA_V$-modules supported at the points $z=0$ and $z=\infty$ (which is to say $x=0$). The choice of coordinates $z$ and $x = 1/z$ induce tangent vectors $\xi$ at $0$ and $\infty$ (which we will henceforth omit from the notation), and identifications $\CM_{M, 0} \rightarrow M$ and $\CM_{M, \infty} \rightarrow M$. The relationship between this setup and the definition of the contragredient dual is encapsulated in Proposition \ref{prop:dual.fancy} below. In its proof we will use the following result from \cite{FHL} (see Lemma 5.2.3, and equation (5.3.1) of \emph{loc. cit.})
\begin{prop}\label{prop:FHL}
The identity
\[
(-t^{-2})^{L_0} e^{t^{-1}L_1} (-t^{-2})^{L_0}  = e^{-t L_1}
\]
holds in $\en(V)[[t, t^{-1}]]$.
\end{prop}

\begin{prop}\label{prop:dual.fancy}
Let $V$ be a conformal vertex algebra, $M$ an admissible $V$-module, and $M^\vee$ its contragredient. Under the induced identifications $\CM_{M, 0} \rightarrow M$ and $\CM_{M^\vee, \infty} \rightarrow M^\vee$, we have
\[
[\alpha_{\CM_{M^\vee, \infty}}(\sigma)\varphi](m) = - \varphi(\alpha_{\CM_{M, 0}}(\sigma)m).
\]
for all $\sigma \in \Gamma(\C^\times, \CA_V)$, $m \in M$ and $\varphi \in M^\vee$.
\end{prop}

\begin{proof}
It suffices to take $z^n (z, a) \, dz$ for $\sigma$, for some $a \in V$ and $n \in \Z$. We first write $\sigma$ in terms of the trivialisation associated with $x$. At the point $x = x_0$ we have local coordinates $x - x_0$ and $z - x_0^{-1}$ related by $z - x_0^{-1} = g(x - x_0)$ where
\[
g(t) = -x_0^{-1} t (x_0+t)^{-1} = e^{-x_0^{-1} t^2 \partial_t} \cdot \left( (-x_0^2)^{-t\partial_t} \cdot t \right).
\]
The linear automorphism of $V$ corresponding to $g$ is
\[
R(g) = e^{x_0^{-1} L_1} (-x_0^2)^{L_0},
\]
and we have $(z - x_0^{-1}, a) = (x - x_0, R(g)a)$. Since $dz = -x^{-2} dx$, we obtain
\begin{align}\label{eq:z-to-x}
\sigma = -x^{-n-2} (x, e^{x^{-1} L_1} (-x^2)^{L_0} a) \, dx.
\end{align}

We now turn to the proof of the proposition. As per formula \eqref{eq:point.module.action}, we have
\begin{align*}
[\alpha_{\CM_{M^\vee, \infty}}(\sigma)\varphi](m)
&= -[\res_{x=0} x^{-n-2} Y_{M^\vee}(e^{x^{-1} L_1} (-x^2)^{L_0} a, x) \varphi](m) \\
&= -\varphi \left( \res_{x=0} x^{-n-2} Y(e^{x L_1} (-x^{-2})^{L_0} e^{x^{-1} L_1} (-x^2)^{L_0} a, x^{-1}) m \right) \\
&= - \varphi \left( \res_{x=0} x^{-n-2} Y(a, x^{-1}) m \right),
\end{align*}
where in the first line we have used \eqref{eq:z-to-x}, to pass to the second line we have used Definition \ref{def:dual}, and to pass to the third line we have used Proposition \ref{prop:FHL} above.

But clearly
\begin{align*}
\varphi \left( \res_{x=0} x^{-n-2} Y(a, x^{-1}) m \right) = \varphi(a_{(n)}m) = \varphi \left( \res_{z=0} z^{n} Y(a, z) m \right) = \varphi(\alpha_{\CM_{M, 0}}(\sigma)m)
\end{align*}
and so we are done.
\end{proof}

We will abuse notation slightly and write
\begin{align*}
\alpha_{\CM_{M^\vee, \infty}}(\sigma)\varphi
= \res_{z=\infty}f(z) Y_{M^\vee}(\dualfactor(z)a, z^{-1}) \varphi
\end{align*}
where $\dualfactor(z) = z^{2} e^{z L_1} (-z^{-2})^{L_0}$ and $\res_{z=\infty}$ means to substitute $x=z^{-1}$ and apply $\res_{x=0}$. Then the statement of Proposition \ref{prop:dual.fancy} becomes
\begin{align}\label{eq:dual}
\res_{z=\infty} f(z) \left(Y_{M^{\vee}}(\dualfactor(z)a,z^{-1})\varphi \right)(m)
=-\varphi(\res_{z=0}f(z) Y_{M}(a,z)m).
\end{align}

\begin{rem}\label{rem:mobius}
Let $X$ be a curve with a choice of projective structure, i.e., atlas relative to which transition functions are restrictions of M\"{o}bius transformations. This data yields a principal $PGL_2(\C)$-bundle on $X$ with flat connection, and a reduction of the structure group to a fixed Borel $B$ \cite[Section 8.2]{FBZ}. We denote this principal $B$-bundle as $\CP$. In particular $X = \bbP^1$ comes equipped with such structures.

A vertex algebra $V$ is said to be M\"{o}bius, roughly speaking, if it possesses a conformal grading and a compatible action of $\mathfrak{sl}_2 = \left<L_{-1}, L_0, L_1\right>$, see {\cite[Definition 4.9]{KacVA}} for details. In particular a conformal vertex algebra is M\"{o}bius.

Exponentiating the action of $L_0$ and $L_1$ on the M\"{o}bius vertex algebra $V$ and applying the associated bundle construction as in \cite{FBZ} produces a left $\CD_X$-module $\CV = \CP \times_B V$ and a chiral algebra structure on the corresponding right $\CD_X$-module $\CV \otimes \omega_X$, with the chiral operation locally given by the vertex algebra state-field correspondence \cite{BD}. All the results proved below go through at this greater level of generality.
\end{rem}

\section{Chiral homology}\label{sec:homology}

Let $A$ and $C$ be a pair of admissible $V$-modules. We now introduce a chain complex $C_{\bullet}$ that will compute chiral homology $H_{1}^{\textup{ch}}(\bbP^1, \CA_{V}, \CM_{C,0},\CM_{A^{\vee},\infty})$.

We write $(U_1, z)$ and $(U_2, x)$ for the affine charts of $\bbP^1$ as in Section \ref{sec:chiral.algebras}, and we write $X$ for $U_1 \cap U_2 = \bbP^{1} \setminus \{0,\infty\}$. Let us denote by $\Gamma_n$ the coordinate ring of the configuration space $F(X,n)=\{ (x_1, \dots, x_n) \in X^{n} \colon x_i \neq x_j \text{ for all } i\neq j \}$, i.e.,
\[
\Gamma_n=\text{Spec }\C[z_i, z_i^{-1},(z_i-z_j)^{-1} \colon 1 \leq i < j \leq n ].
\]
For each $i = 1,\dots,n$ we will denote by $\partial_i$ the derivation of $\Gamma_n$ defined to be the pullback of $\partial_z$ along the $i^{\text{th}}$ projection map.

We now define
\[
\widetilde{C}_n=\Gamma_n \otimes V^{\otimes n} \otimes A^{\vee} \otimes C\quad \text{and}\quad C_n=\widetilde{C}_n / \sum_{i=1}^{n} (\partial_i + T_i)\widetilde{C}_n,
\]
where $T_i$ denotes the action of the translation operator on the $i^{\text{th}}$ tensor factor of $V^{\otimes n}$.

The differential $d_n \colon C_n \rightarrow C_{n-1}$ will be defined first as a map $d_n \colon \widetilde{C}_n \rightarrow \widetilde{C}_{n-1}$, which in fact will descend to the quotient $C_{\bullet}$ and endow it with a chain complex structure (see \cite{vEH2023}).

We define
\[
d_n = \sum_{1 \leq i < j \leq n} (-1)^{n-i} d^{(ij)} + \sum_{1 \leq i \leq n} (-1)^{n-i}\left(p_C^{(i)}+p_{A^{\vee}}^{(i)}\right),
\]
where
\begin{align*}
& d^{(ij)}( f(z_1,\dots,z_n) a^1 \otimes \cdots \otimes a^n \otimes \varphi \otimes c) \\
&= {} \res_{z_i=z_j} f(z_1,\dots,z_n)a^1 \otimes \cdots \otimes \widehat{a}^{i} \otimes \cdots
\otimes Y(a^i,z_i-z_j)a^j \otimes \cdots \otimes a^n \otimes \varphi \otimes c \big|_{(z_1,\dots,\widehat{z}_i,\dots,z_n)=(z_1,\dots,z_{n-1})},
\end{align*}
\begin{align*}
& p_{C}^{(i)}( f(z_1,\dots,z_n) a^1 \otimes \cdots \otimes a^n \otimes \varphi \otimes c) \\
= {} & \res_{z_i=0}f(z_1,\dots,z_n)a^1 \otimes \cdots \otimes \widehat{a}^{i} \otimes \cdots
\otimes a^n \otimes \varphi \otimes Y_{C}(a^i,z_i) c \big|_{(z_1,\dots,\widehat{z}_i,\dots,z_n)=(z_1,\dots,z_{n-1})}
\end{align*}
and
\begin{align*}
& p_{A^{\vee}}^{(i)}(f(z_1,\dots,z_n) a^1 \otimes \cdots \otimes a^n \otimes \varphi \otimes c)\\
= {} & \res_{z_i=\infty} f(z_1,\dots,z_n)a^1 \otimes \cdots \otimes \widehat{a}^{i} \otimes \cdots \otimes a^n \otimes Y_{A^{\vee}}(\dualfactor(z_i)a^i,z_i^{-1}) \varphi \otimes c \big|_{(z_1,\dots,\widehat{z}_i,\dots,z_n)=(z_1,\dots,z_{n-1})}.
\end{align*}
In the last of these definitions the symbols $\res_{z_i=\infty}$ and $\dualfactor(z_i)$ are understood as in equation (\ref{eq:dual}) and preceding comments, and the residue is to be taken after expanding the function $f(z_1, \ldots, z_n)$ in the domain $|z_i| > |z_j|$ for $j\neq i$.

In particular, we have
\begin{align}\label{eq:d1}
\begin{split}
d_1(f(z)a \otimes \varphi \otimes c) = {} & \res_{z=\infty} f(z)Y_{A^{\vee}}(\dualfactor(z)a,z^{-1})\varphi \otimes c \\
&+\res_{z=0} f(z)\varphi \otimes Y_{C}(a,z)c
\end{split}
\end{align}
and
\begin{align}\label{eq:d2}
\begin{split}
d_2(f(z_1,z_2)a \otimes b \otimes \varphi \otimes c) = &-\res_{z_1=z_2}f(z_1,z_2)Y(a,z_1-z_2)b \otimes \varphi \otimes c \big|_{z_2=z_1}\\
&-\res_{z_1=\infty}f(z_1,z_2)b \otimes Y_{A^{\vee}}(\dualfactor(z_1) a,z_1^{-1}) \varphi \otimes c \big|_{z_2=z_1}\\
&-\res_{z_1=0}f(z_1,z_2)b \otimes \varphi \otimes Y_{C}(a,z_1)  c \big|_{z_2=z_1}\\
&+\res_{z_2=\infty}f(z_1,z_2)a \otimes Y_{A^{\vee}}(\dualfactor(z_2) b,z_2^{-1}) \varphi \otimes c \big|_{z_1=z_1}\\
&+\res_{z_2=0}f(z_1,z_2)a \otimes \varphi \otimes Y_{C}(b,z_2)  c \big|_{z_1=z_1}.
\end{split}
\end{align}

\begin{rem}
By Proposition \ref{prop:dual.fancy} and \eqref{eq:d1} we have
\begin{align*}
d_1(z^n a \otimes \varphi \otimes c) = \varphi \otimes a_{(n)}c - \varphi \circ a_{(n)} \otimes c.
\end{align*}
\end{rem}

The homology groups $H_{n}(C_{\bullet})$, which we denote $H_{n}^{\text{ch}}(\bbP^1,\CA_{V},\CM_{C,0},\CM_{A^{\vee},\infty})$, coincide with those defined by Beilinson and Drinfeld in \cite[Chapter 4]{BD} for $n=0$ and $n=1$. See \cite[Section 7]{vEH2021}.

\section{The main theorem}

Let $V$ be a conformal vertex algebra and let $A$ and $C$ be two admissible $V$-modules. The goal of this section is to prove the following result.
\begin{thm}\label{equivalence_extensions_H1}
There exists an isomorphism
\[
\Psi: \Ext^1(C, A) \overset{\simeq}{\longrightarrow} H_{1}^{\textup{ch}}(\bbP^1, \CA_{V},\CM_{C,0},\CM_{A^{\vee},\infty})^{*}.
\]
\end{thm}

\begin{proof}
We begin by describing the construction of $\Psi$. Let
\begin{align*}
\xymatrix{
0 \ar@{->}[r] & A \ar@{->}[r]^{i} & B \ar@{->}[r]^{p} & C \ar@{->}[r] \ar@/^1.0pc/@[][l]^{s} & 0 \\
}
\end{align*}
be an extension of admissible $V$-modules of $C$ by $A$, for which we have chosen an arbitrary section $s$ of $B \xlongrightarrow{p} C$ at the level of graded vector spaces.

Note that for any given $a \in V$ and $c \in C$, the expression
\begin{equation}\label{definition_cocycle_xi}
\xi(a,z)c := Y_{B}(a,z)s(c) - s\left(Y_{C}(a,z)c\right)    
\end{equation}
lies in $(\ker{p})((z))\simeq A((z))$, which allows us to define a map $\widetilde{\psi}_{B} \colon \Gamma_1 \otimes V \otimes A^{\vee} \otimes C \rightarrow \C$ as
\begin{equation}\label{definition_psi_from_xi}
\widetilde{\psi}_{B}(f(z)a \otimes \varphi \otimes c) = \varphi(\res_{z=0}f(z)\xi(a,z)c).
\end{equation}
If we considered another section $s'\colon C \rightarrow B$ then we may use it to define $\widetilde{\psi}_{B}'$ in just the same way as we did for $s$. Although $\widetilde{\psi}_{B}'$ and $\widetilde{\psi}_{B}$ are distinct elements of $(\Gamma_1 \otimes V \otimes A^{\vee} \otimes C)^*$, their restrictions to the subspace $\ker{d_1}$ coincide, as we shall now show.

Let us write $\eta = s - s'$. We see that $p \circ \eta = 0$ and so we may think of $\eta$ as an element of $\Hom_\C{(C,A)}$. Let $f(z) \cdot a \otimes \varphi \otimes c \in \ker{d_1}$. Then
\[
0=\res_{z=\infty}f(z) Y_{A^{\vee}}(\dualfactor(z)a,z^{-1})\varphi \otimes c + \res_{z=0}f(z)\varphi \otimes Y_{C}(a,z)c.
\]
If we apply the composition
\[
A^{\vee} \otimes C \xrightarrow{\text{id} \otimes \eta} A^{\vee} \otimes A \xlongrightarrow{\text{ev}} \C
\]
to that equality, we obtain
\begin{align*}
0&=\res_{z=\infty}f(z) \left(Y_{A^{\vee}}(\dualfactor(z)a,z^{-1})\varphi \right)(\eta(c)) + \res_{z=0}f(z)\varphi (\eta(Y_{C}(a,z)c))\\
&=-\varphi(\res_{z=0}f(z) Y_{A}(a,z)\eta(c)) + \varphi (\res_{z=0}f(z)\eta(Y_{C}(a,z)c))\\
&=-\varphi(\res_{z=0}f(z) Y_{B}(a,z)\eta(c)) + \varphi (\res_{z=0}f(z)\eta(Y_{C}(a,z)c))\\
&=-\varphi(\res_{z=0}f(z) Y_{B}(a,z)s(c)) + \varphi (\res_{z=0}f(z)s(Y_{C}(a,z)c))\\
&\phantom{=}+\varphi(\res_{z=0}f(z) Y_{B}(a,z)s'(c)) - \varphi (\res_{z=0}f(z)s'(Y_{C}(a,z)c))\\
&=-\widetilde{\psi}_{B}(f(z)a \otimes \varphi \otimes c) + \widetilde{\psi}_{B}'(f(z)a \otimes \varphi \otimes c).
\end{align*}

We will now show the following facts:
\begin{align}\label{eq:fact1}
\widetilde{\psi}_{B}((\partial_z+T)(\Gamma_1 \otimes V \otimes A^{\vee} \otimes C))=0
\end{align}
and
\begin{align}\label{eq:fact2}
\widetilde{\psi}_{B}(\img{d_2})=0,
\end{align}
from which it follows that $\widetilde{\psi}_{B}$ induces a well-defined element $\psi_{B} \in H_{1}^{\textup{ch}}(\bbP^1, \CA_{V},\CM_{C,0},\CM_{A^{\vee},\infty})^{*}$.

For \eqref{eq:fact1} we simply note that for any $f(z) a \otimes \varphi \otimes c \in \Gamma_1 \otimes V \otimes A^{\vee} \otimes C$ we have
\begin{align*}
\widetilde{\psi}_{B}&(\partial_z f(z) a \otimes \varphi \otimes c + f(z) Ta \otimes \varphi \otimes c)\\
&=\varphi (\res_{z=0}\partial_z f(z) \xi(a,z)c) + \varphi(\res_{z=0}f(z) \xi(Ta,z)c)\\
&=\varphi (\res_{z=0}\partial_z f(z) Y_{B}(a,z)s(c) - \res_{z=0}\partial_z f(z) s(Y_{C}(a,z)c))\\
&\phantom{=}+\varphi (\res_{z=0} f(z) Y_{B}(Ta,z)s(c) - \res_{z=0}f(z) s(Y_{C}(Ta,z)c))\\
&=\varphi ((\res_{z=0}\partial_z f(z) Y_{B}(a,z) + \res_{z=0} f(z) Y_{B}(Ta,z))s(c))\\
&\phantom{=}-\varphi (s(\res_{z=0}\partial_z f(z) Y_{C}(a,z)c + \res_{z=0}f(z) Y_{C}(Ta,z)c))\\
&=0.
\end{align*}
To prove \eqref{eq:fact2} we first compute using Proposition \ref{prop:dual.fancy}, or rather equation (\ref{eq:dual}),
\begin{align}
\widetilde{\psi}_{B}&(d_2(f(z_1,z_2)a \otimes b \otimes \varphi \otimes c))\nonumber\\
=&-\widetilde{\psi}_{B}(\res_{z_1=z_2}f(z_1,z_2)Y(a,z_1-z_2)b \otimes \varphi \otimes c \big|_{z_2=z_1})\nonumber\\
&-\widetilde{\psi}_{B}(\res_{z_1=\infty}f(z_1,z_2)b \otimes Y_{A^{\vee}}(\dualfactor(z_1)a,z_1^{-1}) \varphi \otimes c \big|_{z_2=z_1})\nonumber\\
&-\widetilde{\psi}_{B}(\res_{z_1=0}f(z_1,z_2)b \otimes \varphi \otimes Y_{C}(a,z_1)  c \big|_{z_2=z_1})\nonumber\\
&+\widetilde{\psi}_{B}(\res_{z_2=\infty}f(z_1,z_2)a \otimes Y_{A^{\vee}}(\dualfactor(z_2)b,z_2^{-1}) \varphi \otimes c \big|_{z_1=z_1})\nonumber\\
&+\widetilde{\psi}_{B}(\res_{z_2=0}f(z_1,z_2)a \otimes \varphi \otimes Y_{C}(b,z_2)  c \big|_{z_1=z_1})\nonumber\\
=&-\res_{z_2=0}\res_{z_1=z_2}f(z_1,z_2)\varphi(\xi(Y(a,z_1-z_2)b,z_2) c)\nonumber\\
&-\res_{z_2=0}\res_{z_1=\infty}f(z_1,z_2)(Y_{A^{\vee}}(\dualfactor(z_1)a,z_1^{-1}) \varphi)( \xi(b,z_2)c)\nonumber\\
&-\res_{z_2=0}\res_{z_1=0}f(z_1,z_2)\varphi(\xi(b,z_2)(Y_{C}(a,z_1)c))\nonumber\\
&+\res_{z_1=0}\res_{z_2=\infty}f(z_1,z_2)(Y_{A^{\vee}}(\dualfactor(z_2)b,z_2^{-1}) \varphi)(\xi(a,z_1)c)\nonumber\\
&+\res_{z_1=0}\res_{z_2=0}f(z_1,z_2)\varphi(\xi(a,z_1)(Y_{C}(b,z_2)c))\nonumber\\
=&-\res_{z_2=0}\res_{z_1=z_2}f(z_1,z_2)\varphi(\xi(Y(a,z_1-z_2)b,z_2) c)\label{psi_kills_differentials}\\
&+\res_{z_2=0}\res_{z_1=0}f(z_1,z_2)\varphi(Y_{A}(a,z_1)\xi(b,z_2)c)\nonumber\\
&-\res_{z_2=0}\res_{z_1=0}f(z_1,z_2)\varphi(\xi(b,z_2)(Y_{C}(a,z_1)c))\nonumber\\
&-\res_{z_1=0}\res_{z_2=0}f(z_1,z_2)\varphi(Y_{A}(b,z_2)\xi(a,z_1)c)\nonumber\\
&+\res_{z_1=0}\res_{z_2=0}f(z_1,z_2)\varphi(\xi(a,z_1)(Y_{C}(b,z_2)c)).\nonumber
\end{align}

We now check that vanishing of this expression for all $f(z_1,z_2)a\otimes b \otimes \varphi \otimes c \in \Gamma_{2} \otimes V^{\otimes 2} \otimes A^{\vee} \otimes C$ is equivalent to a ``cocycle condition'' on $\xi(\cdot,z)$, which in turn derives from the Borcherds identity in the $V$-modules $B$ and $C$.

Since we started with a short exact sequence $0 \rightarrow A \rightarrow B \rightarrow C \rightarrow 0$ of $V$-modules, we know that the Borcherds identity holds in $B$. In particular, for any given $c \in C$ we have 
\begin{align*}
0=&-\res_{z_2=0}\res_{z_1=z_2}f(z_1,z_2)Y_B(Y(a,z_1 - z_2)b,z_2)s(c)\\
&-\res_{z_1=0}\res_{z_2=0}f(z_1,z_2)Y_{B}(b,z_2)Y_B(a,z_1)s(c)\\
&+\res_{z_1=0}\res_{z_2=0}f(z_1,z_2)Y_{B}(a,z_1)Y_B(b,z_2)s(c)\\
=&-\res_{z_2=0}\res_{z_1=z_2}f(z_1,z_2)(\xi(Y(a,z_1-z_2)b,z_2)c+s(Y_C(Y(a,z_1 - z_2)b,z_2)c))\\
&-\res_{z_1=0}\res_{z_2=0}f(z_1,z_2)Y_{B}(b,z_2)(\xi(a,z_1)c+s(Y_C(a,z_1)c))\\
&+\res_{z_1=0}\res_{z_2=0}f(z_1,z_2)Y_{B}(a,z_1)(\xi(b,z_2)c+s(Y_C(b,z_2)c))\\
=&-\res_{z_2=0}\res_{z_1=z_2}f(z_1,z_2)(\xi(Y(a,z_1-z_2)b,z_2)c+s(Y_C(Y(a,z_1 - z_2)b,z_2)c))\\
&-\res_{z_1=0}\res_{z_2=0}f(z_1,z_2)(Y_{A}(b,z_2)\xi(a,z_1)c+\xi(b,z_2)Y_{C}(a,z_1)c+s(Y_C(b,z_2)Y_C(a,z_1)c))\\
&+\res_{z_1=0}\res_{z_2=0}f(z_1,z_2)(Y_{A}(a,z_1)\xi(b,z_2)c+\xi(a,z_1)Y_{C}(b,z_2)c+s(Y_C(a,z_1)Y_C(b,z_2)c))
\end{align*}
\vspace{-0.75cm}
\begin{align}
=&-\res_{z_2=0}\res_{z_1=z_2}f(z_1,z_2)\xi(Y(a,z_1-z_2)b,z_2)c\hspace{6cm}\label{cocycle_condition_xi}\\
&+\res_{z_1=0}\res_{z_2=0}f(z_1,z_2)Y_{A}(a,z_1)\xi(b,z_2)c\nonumber\\
&-\res_{z_1=0}\res_{z_2=0}f(z_1,z_2)\xi(b,z_2)Y_{C}(a,z_1)c\nonumber\\
&-\res_{z_1=0}\res_{z_2=0}f(z_1,z_2)Y_{A}(b,z_2)\xi(a,z_1)c\nonumber\\
&+\res_{z_1=0}\res_{z_2=0}f(z_1,z_2)\xi(a,z_1)Y_{C}(b,z_2)c\nonumber\\
&-s(\res_{z_2=0}\res_{z_1=z_2}f(z_1,z_2)Y_C(Y(a,z_1 - z_2)b,z_2)c)\nonumber\\
&-s(\res_{z_1=0}\res_{z_2=0}f(z_1,z_2)Y_{C}(b,z_2)Y_C(a,z_1)c)\nonumber\\
&+s(\res_{z_1=0}\res_{z_2=0}f(z_1,z_2)Y_{C}(a,z_1)Y_C(b,z_2)c).\nonumber
\end{align}
The sum of the last three terms in equation (\ref{cocycle_condition_xi}) vanishes due to the Borcherds identity for $C$. The remaining five terms become exactly the same (and in the same order) as the ones that appear in (\ref{psi_kills_differentials}) after applying $\varphi$. In particular, since we assumed that $B$ is a $V$-module, we know that equation (\ref{cocycle_condition_xi}) holds, and therefore the expression (\ref{psi_kills_differentials}) is zero, so we have proved \eqref{eq:fact2}. It follows that $\widetilde{\psi}_B$ descends to a well defined class $\psi_{B} \in H_{1}^{\textup{ch}}(\bbP^1, \CA_{V},\CM_{C,0},\CM_{A^{\vee},\infty})^{*}$.

Now, let $0 \rightarrow A \xrightarrow{i'} B' \xrightarrow{p'} C \rightarrow 0$ be another extension of admissible $V$-modules of $C$ by $A$, isomorphic to the original one, meaning that we have an isomorphism of $V$-modules $\alpha : B \rightarrow B'$ such that the diagram
\begin{align*}
\xymatrix{
0 \ar@{->}[r] & A \ar@{->}[r] \ar@{->}[d]_{\text{id}} & B \ar@{->}[r] \ar@{->}[d] & C \ar@{->}[r] \ar@{->}[d]^{\text{id}} & 0 \\
0 \ar@{->}[r] & A \ar@{->}[r] & B' \ar@{->}[r] & C \ar@{->}[r] & 0 \\
}
\end{align*}
commutes. We use the splitting $s'=\alpha \circ s: C \rightarrow B'$ to define $\xi'(\cdot,z)$ in the same way as before, obtaining the class $\psi_{B'}$. Now we have
\begin{align*}
\xi'(a,z)c&=i'^{-1}(Y_{B'}(a,z)\alpha(s(c)) - \alpha(s(Y_{C}(a,z)c))\\
&=i'^{-1}(\alpha(Y_{B}(a,z)s(c)) - \alpha(s(Y_{C}(a,z)c)))\\
&=i^{-1}(Y_{B}(a,z)s(c)-s(Y_{C}(a,z)c)))\\
&=\xi(a,z)c,
\end{align*}
which means that $\psi_{B}=\psi_{B'}$.

This completes the construction of the map
\[
\Psi: \Ext^1(C, A) \rightarrow H_{1}^{\textup{ch}}(\bbP^1, \CA_{V},\CM_{C,0},\CM_{A^{\vee},\infty})^{*},
\]
which assigns to $[B] \in \Ext^1(C, A)$ the element $\Psi([B])=\psi_{B}\in H_{1}^{\textup{ch}}(\bbP^1,\CA_{V},\CM_{C,0},\CM_{A^{\vee},\infty})^{*}$. It remains to show that $\Psi$ is an isomorphism.

First we prove surjectivity of $\Psi$. Let $\psi \in H_{1}^{\textup{ch}}(\bbP^1,\CA_{V},\CM_{C,0},\CM_{A^{\vee},\infty})^{*}$. We shall construct a $V$-module structure on the vector space $B:= A \oplus C$ in such a way that $0 \rightarrow A \rightarrow B \rightarrow C \rightarrow 0$ is an extension of $V$-modules and $\Psi([B])=\psi$.

We set $Y_{B}(a,z)(m,0)=(Y_{A}(a,z)m,0)$ for all $m \in A$, while for the action of $Y_B(a,z)$ on an element $(0,c)$ for $c\in C$ we set
\[
Y_B(a,z)(0,c)=(\xi(a,z)c,Y_{C}(a,z)c)
\]
for some choice
\[
\xi(a,z)c = \sum_{n\in \Z}\xi_{n}(a)(c)z^{-n-1}\in A((z)).
\]
We now work to define $\xi(-, z)$ in such a way that $Y_B(-, z)$ equipes $B$ with a $V$-module structure, and furthermore that equation \eqref{definition_cocycle_xi} holds for the obvious section $s: C \rightarrow B$, $c \mapsto (0,c)$.

We now choose a retraction $p_{\ker{d_1}}$ of the inclusion $\ker{d_1} \hookrightarrow \Gamma_1 \otimes V \otimes A^{\vee} \otimes C$ at the level of vector spaces, and we define $\xi(-,z)$ by
\begin{equation}\label{definition_xi_from_psi}
\varphi(\xi_{n}(a)(c))=\psi(p_{\ker{d_1}}(z^n a \otimes \varphi \otimes c))    
\end{equation}
for all $\varphi \in A^{\vee}$. This uniquely determines $\xi(-, z)$. Note that the relation
\begin{equation}\label{recovering_psi_from_xi}
\psi(f(z)a \otimes \varphi \otimes c) = \varphi (\res_{z=0}f(z)\xi(a,z)c) 
\end{equation}
holds for all $f(z) a \otimes \varphi \otimes c \in \ker{d_1}$. Therefore, once we have confirmed that $Y_B(-,z)$ defines a $V$-module structure on $B$, comparing the equations (\ref{definition_psi_from_xi}) and (\ref{recovering_psi_from_xi}) will yield $\Psi([B])=\psi$.

In order to prove that $Y_{B}({\vac},z)=\id_{B}$, it suffices to see that $\xi({\vac},z)=0$, and this in turn follows from
\[
\psi(d_2(z_2^{n}(z_1-z_2)^{-1} {\vac} \otimes {\vac} \otimes \varphi \otimes c))=0
\]
which can be seen from formula \eqref{psi_kills_differentials}. On the other hand, we need to check that the Borcherds identity holds in $B$. Since we know it holds in $A$, it suffices to check that equation (\ref{cocycle_condition_xi}) holds. But this follows from the fact that $\psi(\img{d_2})=0$, i.e. from the fact that expression (\ref{psi_kills_differentials}) vanishes. Therefore $B$ is indeed a $V$-module, and this completes the proof that $\Psi$ is surjective.

Finally we check that $\ker{\Psi}=0$. Let $B$ be an extension of $C$ by $A$ for which $\psi_{B}=0$, and $s: C \rightarrow B$ a section. The restriction of $\widetilde{\psi}_{B} : \Gamma_1 \otimes V \otimes A^{\vee} \otimes C \rightarrow \C$ to the subspace $\ker{d_1}$ vanishes, so we have a well defined map
\[
\Omega : \img{d_1} \xlongrightarrow{\simeq} \Gamma_1 \otimes V \otimes A^{\vee} \otimes C / \ker{d_1} \xlongrightarrow{\widetilde{\psi}_{B}} \C,
\]
i.e.
\begin{equation}\label{definition_omega}
\widetilde{\psi}_{B}(f(z) a \otimes \varphi \otimes c)=\Omega(d_1(f(z) a \otimes \varphi \otimes c))
\end{equation}
for all $f(z) a \otimes \varphi \otimes c \in \Gamma_1 \otimes V \otimes A^{\vee} \otimes C$. We choose an arbitrary extension of $\Omega$ to a linear map $\Omega : A^{\vee} \otimes C \rightarrow \C$, defining it as zero on some subspace of $A^{\vee} \otimes C$ complementary to $\img{d_1}$. We then restrict the domain of $\Omega$ in each graded piece $(A^\vee \otimes C)_{2n}$ to its subspace $A_n^* \otimes C_n$, which yields an element $\Omega^\circ$ in the space $\Pi_{n \in \Z_+} (A_n^* \otimes C_n)^*$.

For simplicity, let us assume that the graded pieces of $A$ and $C$ are finite-dimensional. We discuss how to adapt the argument to the infinite-dimensional setting after the proof.

We have an isomorphism
\begin{equation}\label{eq:comp.iso}
\theta \colon \Hom^{\text{gr}}(C,A) = \Pi_{n \in \Z_+}\Hom(C_n,A_n) \xlongrightarrow{\simeq} \Pi_{n \in \Z_+} (A_n^* \otimes C_n)^{*}\\
\end{equation}

Let us define $\eta : C \rightarrow A$ as the graded map given by the inverse image $\theta^{-1}(\Omega^\circ)$. Then we can write $\Omega(\varphi \otimes c)=\theta(\eta)(\varphi \otimes c) = \varphi(\eta(c))$ for any pair of vectors $\varphi \in A^\vee$ and $c \in C$ homogeneous of the same degree. In particular, for all $m \in \Z$, $r,n \in \Z_+$, $a \in V_r$, $\varphi \in A_{r+n-m-1}^*$ and $c \in C_n$ the summands of the vector $d_1(z^m a \otimes \varphi \otimes c)= \varphi \otimes a_{(m)}c - \varphi \circ a_{(m)} \otimes c$ have the required property, i.e. $\text{deg}(\varphi)=\text{deg}(a_{(m)}c)$ and $\text{deg}(\varphi \circ a_{(m)})=\text{deg}(c)$. This allows us to compute
\begin{align*}
&\Omega(d_1(z^m a \otimes \varphi \otimes c))\\
&=\theta(\eta)(\res_{z=\infty}z^m Y_{A^{\vee}}(\dualfactor(z)a,z^{-1})\varphi \otimes c +\res_{z=0}z^m \varphi \otimes Y_{C}(a,z)c)\\
&=\res_{z=\infty}z^m (Y_{A^{\vee}}(\dualfactor(z)a,z^{-1})\varphi)\eta(c) +\res_{z=0}z^m \varphi(\eta(Y_{C}(a,z)c))\\
&=-\res_{z=0}z^m \varphi(Y_{A}(a,z)\eta(c))+\res_{z=0}z^m \varphi(\eta(Y_{C}(a,z)c))\\
&=\varphi(\res_{z=0}z^m (-Y_{A}(a,z)\eta(c)+\eta(Y_{C}(a,z)c))).
\end{align*}
On the other hand, using definitions (\ref{definition_cocycle_xi}) and (\ref{definition_psi_from_xi}) we get
\begin{align*}
\widetilde{\psi}_{B}(z^m a\otimes \varphi \otimes c) &= \varphi(\res_{z=0}z^m \xi(a,z)c)\\
&= \varphi(\res_{z=0}z^m i^{-1}(Y_{B}(a,z)s(c)-s(Y_{C}(a,z)c))).
\end{align*}

Now plugging these results in (\ref{definition_omega}) we obtain that the equality
\[
\varphi(\res_{z=0}z^m (-Y_{A}(a,z)\eta(c)+\eta(Y_{C}(a,z)c)))= \varphi(\res_{z=0}z^m i^{-1}(Y_{B}(a,z)s(c)-s(Y_{C}(a,z)c)))
\]
holds for all $m \in \Z$, $r,n \in \Z_+$, $a \in V_r$, $\varphi \in A_{r+n-m-1}^*$ and $c \in C_n$. Since the argument of $\varphi$ in both sides is homogeneous of degree $r+n-m-1=\text{deg}(\varphi)$, we can deduce that for all $f(z) \in \Gamma_1$, $a\in V$ and $c\in C$ we have
\begin{equation}\label{tilde_s_is_morphism}
\res_{z=0}f(z)(-Y_{A}(a,z)\eta(c)+\eta(Y_{C}(a,z)c))=\res_{z=0}f(z)i^{-1}(Y_{B}(a,z)s(c)-s(Y_{C}(a,z)c)).    
\end{equation}

Therefore, if we define the map $\widetilde{s}=s+i \eta : C \rightarrow B$ we can rewrite (\ref{tilde_s_is_morphism}) as
\[
\res_{z=0}f(z)\widetilde{s}(Y_C(a,z)c) = \res_{z=0}f(z) Y_{B}(a,z)\widetilde{s}(c)
\]
for all $f(z) \in \Gamma_1$, $a\in V$ and $c\in C$, which means that $\widetilde{s}:C \rightarrow B$ is a morphism of $V$-modules. On the other hand $\widetilde{s}$ is a section of our extension since $p \circ \widetilde{s} = p \circ (s+i \eta) = p \circ s$, and thus the original extension is trivial.
\end{proof}

In the arguments above we have assumed the graded pieces of $A$ and $C$ to be finite-dimensional for ease of exposition. This assumption is not essential, however, as we now explain.

Suppose $A$ and $C$ are admissible $V$-modules. Each component $A_n^*$ of $A^\vee$, being the full dual of a discrete vector space, may be equipped with the linearly compact topology, which has the property that its continuous dual $(A_n^*)^{*\text{, cont}}$ (i.e., the space of continuous linear functionals $A_n^* \to \C$) is isomorphic to the original space $A_n$.

We have an isomorphism
\begin{equation*}
\theta \colon \Hom^{\text{gr}}(C,A) = \Pi_{n \in \Z_+}\Hom(C_n,A_n) \xlongrightarrow{\simeq} \Pi_{n \in \Z_+} (A_n^* \otimes C_n)^{*\text{, cont}}
\end{equation*}
which recovers \eqref{eq:comp.iso} in case the graded pieces of $A$ and $C$ are finite dimensional and hence discrete.

As noted in Remark \ref{rem:contragred.topology}, the action of $V$ on $A^\vee$ is continuous. It follows that the differentials $d_i$ in the chiral complex, and the linear map $\Omega : \text{im}(d_1) \rightarrow \C$ are continuous. Extension of $\Omega$ from $\text{im}(d_1)$ to $A^\vee \otimes C$, setting $\Omega$ equal to zero on a complementary subspace, preserves continuity. We may thus define a graded map $\eta : C \rightarrow A$ as we did in the proof of Theorem \ref{equivalence_extensions_H1}, and indeed the rest of the proof follows as before.

\end{document}